\documentclass[12pt]{amsart}

\usepackage{amsthm}
\usepackage{amsmath,amssymb,amsfonts,marvosym,amscd}
\usepackage{tikz}
\usepackage{tikz-cd}
\usepackage{wasysym} 

\usepackage{pifont} 
\usepackage{graphicx}
\usepackage{psfrag}
\usepackage{datetime}
\usepackage{enumerate}

\usepackage{geometry} 
\usepackage[all]{xy}

\usepackage{hyperref}

\title{Higher gentle algebras}
\author{Jordan McMahon}
\address{Institut f\"ur Mathematik und Wissenschaftliches Rechnen\\ Universit\"at Graz,
Heinrichstrasse 36, A-8010 Graz, Austria}
\begin{document}

\newtheorem{lm}{Lemma}[section]
\newtheorem{prop}[lm]{Proposition}
\newtheorem{conj}[lm]{Conjecture}
\newtheorem{cor}[lm]{Corollary}
\newtheorem{theorem}[lm]{Theorem}

\theoremstyle{definition}
\newtheorem{eg}{Example}
\newtheorem{defn}{Definition}
\newtheorem{remark}{Remark}
\newtheorem{qu}{Question}

\begin{abstract}
We introduce higher gentle algebras. Our definition allows us to determine the singularity categories and subsequently show that higher gentle algebras are Iwanaga-Gorenstein. Under extra assumptions, we show that cluster-tilted algebras (in the sense of Oppermann-Thomas) of higher Auslander algebras of type $A$ are higher gentle.
\end{abstract}

\maketitle
\tableofcontents
\newcommand{\sub}{\underline}

\section{Introduction}

Gentle algebras were introduced in \cite{as}, as a class of the special-biserial algebras introduced in \cite{sw}. Specifically, gentle algebras encompass the tilted algebras of type $A_n$ and $\tilde{A}_n$. Since then, gentle algebras have appeared naturally in many other contexts; prominent sources include triangulations of surfaces \cite{abcp}, \cite{lf}, tilings of surfaces \cite{drs} \cite{csp}, as well as $m$-Calabi-Yau tilted algebras \cite{else} and Brauer graph algebras \cite{schroll}. More general models have been attained recently \cite{bcs}, \cite{ops}, \cite{ppp}.  
Generalisations of special-biserial algebras and gentle exist in the literature, such as almost-gentle algebras \cite{gs2}, skew-gentle algebras \cite{gdp} and special-multiserial algebras \cite{gs}, \cite{vHW}.
As a natural generalisation of gentle algebras to higher Auslander-Reiten theory, we define higher gentle algebras in Definition \ref{genttle}. 

Gentle algebras are Iwanaga-Gorenstein \cite{gr}, and their singularity categories were described by Kalck in \cite{ka}. On the other hand, singularity categories for Nakayama algebras were calculated by Chen-Ye in \cite{cy}.  Gorenstein-projective modules were classified for monomial algebras in \cite{csz}, to some extent unifying the previous results. Related descriptions of singularity categories and Gorenstein projective modules can be found in  \cite{chenquad}, \cite{cgl}, \cite{cl}, \cite{lu}, \cite{lz}, \cite{ring}, \cite{shen}.

The strategy of Chen-Ye is to find simple modules with projective dimension at most one, and then contract an associated arrow.
This follows the work of Chen-Krause \cite{ck}, while contractions of subquivers of type $A$ in a path algebra were considered in \cite{hkp}. The analysis of the simple modules having projective dimension at most one is also of interest for the contraction of arrows as a reduction technique for the finitistic dimension \cite{bp}, \cite{fs}, \cite{gps}, see also the work of Xi on idempotents and finitistic dimension \cite{xi}. 

The singularity categories of higher Nakayama algebras were classified in \cite{mc2}. To extend the strategy for classical Nakayama algebras to higher Nakayama algebras, the goal is to analyse idempotent ideals rather than simple modules - the role of a simple module of projective dimension at most one is subsumed by that of a fabric idempotent. The existence of fabric idempotents allows us to calculate singularity categories for higher gentle algebras.

\begin{theorem}[Theorem \ref{elso}]
Let $A$ be a $d$-gentle algebra. There exists a product of fabric idempotents $f$ such that $fAf$ is a gentle algebra and such that there is an equivalence of categories $$D_{\mathrm{sg}}(A)\cong D_{\mathrm{sg}}(fAf).$$
\end{theorem}

For an algebra $\Lambda$ of type $\vec{A}$, any rigid $\Lambda$ module $M$ gives rise to a gentle algebra $\mathrm{End}_{\Lambda}(M)^\mathrm{op}$ \cite{csp}. This may be generalised.

\begin{theorem}[Corollary \ref{boo}]
Let $\Lambda$ be a $d$-Auslander algebra of type $\vec{A}$ and $T$ a $d$-rigid $\Lambda$-module such that $\mathrm{add}(T)\subseteq\mathcal{C}\subseteq \mathrm{mod}(\Lambda)$, where $\mathcal{C}$ is the canonical $d$-cluster-tilting subcategory. Then $\mathrm{End}_\Lambda(T)^\mathrm{op}$ is a $d$-gentle algebra.
\end{theorem}

The behaviour of higher cluster-tilting subcategories differs significantly from that of module categories. For example, the number of simple modules in the $d$-cluster-tilting subcategory of the module category of a $d$-Auslander algebra of type $\vec{A}$ is independent of the value of $d$. In general, the members of such a $d$-cluster-tilting subcategory do not, however, have a filtration by these simple modules. In this sense, the $d$-cluster-tilting subcategories a $d$-Auslander algebra of type $\vec{A}$ behave like module categories only to some extent; to this same extent we are able to produce higher gentle algebras.

For the following result, let $A^d_n$ be the $d$-Auslander algebra of type $\vec{A}_n$ with $\mathcal{C}\subseteq \mathrm{mod}(A^d_n)$ the canonical $d$-cluster-tilting subcategory and $\mathcal{O}_{A^d_n}$ the $(d+2)$-angulated cluster category of $A^d_n$.

\begin{prop}[Corollary \ref{simplytoogood}]
Let $S$ be a semisimple $A^d_n$-module in $\mathcal{C}$ such that $\mathrm{Ext}^d_{A^d_n}(S,S)=0$. Let $P$ be a basic projective $A^d_n$-module such that $\mathrm{Ext}^d_{A^d_n}(S,P)=0$ and set $T:=P\oplus \tau^{-1}_d(S)$. Then $\mathrm{End}_{\mathcal{O}_{A^d_n}}(T)^\mathrm{op}$ is a $d$-gentle algebra.
\end{prop}
If $T$ is tilting as an $A^d_n$-module, then the algebra $\mathrm{End}_{\mathcal{O}_{A^d_n}}(T)^\mathrm{op}$ is a cluster-tilted algebra in the sense of Oppermann-Thomas. Unfortunately it is not always true that cluster-tilted algebras (in the sense of Oppermann-Thomas) of higher Auslander algebras of type $A$ are higher gentle, and we provide a counterexample in Section \ref{sec3}. 
\section{Background}
Consider a finite-dimensional algebra $A$ over a field $k$, and fix a positive integer $d$. We will assume that $A$ is of the form $kQ/I$, where $kQ$ is the path algebra over some quiver $Q$ and $I$ is an admissible ideal of $kQ$. For two arrows in $Q$ $\alpha:i\rightarrow j$ and $\beta:j\rightarrow k$, we denote their composition as $\beta\alpha:i\rightarrow k$. Let $A^\mathrm{op}$ denote the opposite algebra of $A$.
An $A$-module will mean a finitely-generated left $A$-module; by $\mathrm{mod}(A)$ we denote the category of $A$-modules. The functor $D=\mathrm{Hom}_k(-,k)$ defines a duality; let $\Omega$ be the syzygy functor and set $\tau_d=\tau\circ \Omega^{d-1}$ to be the $d$-Auslander-Reiten translation \cite[Section 1.4]{iy1}.  For an $A$-module $M$, let $\mathrm{add}(M)$ be the full subcategory of $\mathrm{mod}(A)$ composed of all $A$-modules isomorphic to direct summands of finite direct sums of copies of $M$. 
A subcategory $\mathcal{C}$ of $\mathrm{mod}(A)$ is \emph{precovering} if for any $M\in \mathrm{mod}(A)$ there is an object $C_M\in\mathcal{C}$ and a morphism $f:C_M\rightarrow M$ such that for any morphism $X\rightarrow M$ with $X\in \mathcal{C}$ factors through $f$; that there is a commutative diagram:
$$\begin{tikzcd} 
\ &X\arrow{d}\arrow[dotted]{dl}[above]{\exists}\\
C_M\arrow{r}{f}&M
\end{tikzcd}$$
The object $C_M$ is said to be the \emph{right $\mathcal{C}$-approximation of $M$}. 
The dual notion of precovering is \emph{preenveloping}.  A subcategory $\mathcal{C}$ that is both precovering and preenveloping is called \emph{functorially finite}. For a finite-dimensional algebra $A$, a functorially-finite subcategory $\mathcal{C}$ of $\mathrm{mod}(A)$ is a \emph{$d$-cluster-tilting subcategory} \cite[Definition 2.2]{iy1} \cite[Definition 3.14]{jasso} if it satisfies the following conditions :
\begin{align*}
\mathcal{C}&=\{M\in\mathrm{mod}(A)|\mathrm{Ext}^i_A(\mathcal{C},M)=0 \ \forall\ 0<i<d\}.\\
\mathcal{C}&=\{M\in\mathrm{mod}(A)|\mathrm{Ext}^i_A(M,\mathcal{C})=0 \ \forall\ 0<i<d\}.
\end{align*}
If there exists a $d$-cluster-tilting subcategory $\mathcal{C}\subseteq \mathrm{mod}(A)$ and $\mathrm{gl.dim}(A)\leq d$, then $A$ is \emph{$d$-representation finite} in the sense of \cite{io}.
The \emph{dominant dimension} of $A$, $\mathrm{dom.dim}(A)$, is the number $n$ such that for a minimal injective resolution of $A$:
$$0\rightarrow A \rightarrow I_0\rightarrow \cdots \rightarrow I_{n-1}\rightarrow I_{n}\rightarrow \cdots$$ the modules $I_0,\ldots, I_{n-1}$ are projective-injective and $I_{n}$ is not projective. The class of $d$-representation-finite algebras were characterised by Iyama as follows.
\begin{theorem}\cite[Proposition 1.3, Theorem 1.10]{iy2}\label{artheory}
Let $A$ be a finite-dimensional algebra with the property that $\mathrm{gl.dim}(A)\leq d$. Then there is a unique $d$-cluster-tilting subcategory $\mathcal{C}\subseteq\mathrm{mod}(A)$ if and only if $$\mathrm{dom.dim}(\mathrm{End}_A(M)^\mathrm{op})\geq d+1 \geq\mathrm{gl.dim}( \mathrm{End}_A(M)^\mathrm{op})$$ where $M$ is an additive generator of the subcategory 
$$\mathcal{C}=\mathrm{add}(\{\tau_d^i(DA)|i>0\})\subseteq\mathrm{mod}(A).$$
\end{theorem}
Following Theorem \ref{artheory}, let $\Gamma$ be a finite-dimensional algebra satisfying $$\mathrm{gl.\ dim.}(\Gamma)\leq d+1 \leq \mathrm{dom.\ dim.} (\Gamma)$$ for some positive integer $d\geq 1$.   Then $\Gamma$ is said to be a \emph{$d$-Auslander algebra}. 
An algebra $A$ is called an \emph{$n$-Iwanaga-Gorenstein algebra} if it satisfies the following axioms:

\begin{enumerate} 
\item $\mathrm{inj.dim}_A(A)\leq n$,
\item $\mathrm{proj.dim}_A(DA)\leq n.$
\end{enumerate}

An $A$-module $M$ is said to be \emph{Gorenstein projective} (also referred to as \emph{maximal Cohen-Macaulay} in the literature, most notably in \cite{buch}) if $\mathrm{Ext}^i_A(M,A)=0$ for all $i>0$. The class of Gorenstein projective modules is denoted $\mathrm{GP}(A)$. Likewise, define a module $M$ to be \emph{Gorenstein injective} if $\mathrm{Ext}^i_A(DA,M)=0$ for all $i>0$, and denote by $\mathrm{GI}(A)$ the class of Gorenstein injective modules. For further information about Gorenstein homological algebra, we refer to \cite{chen}.

Let $D^b(A)$ denote the bounded derived category of $\mathrm{mod}(A)$. A complex of $A$-modules is said to be {\em perfect} if it is isomorphic in $D^b(A)$ to a finite complex of finitely generated projective $A$-modules. This gives a full subcategory of $D^b(A)$, denoted by $D_{\mathrm{perf}}^b(A)$. The \emph{singularity category} $D_\mathrm{sg}(A)$ is defined as the Verdier quotient of $D^b(A)/D^b_{\mathrm{perf}}(A)$ \cite{buch} \cite{orlov}. The following theorem is a classical result of Buchweitz.

\begin{theorem}\cite[Theorem 4.4.1]{buch}\label{buch}
Let $A$ be an $n$-Iwanaga-Gorenstein algebra. Then there is an equivalence of (triangulated) categories:
$$\underline{\mathrm{GP}}(A)\cong D_{\mathrm{sg}}(A).$$
\end{theorem}

\subsection{Tilting theory and $(d+2)$-angulated categories}
Let $A$ be a finite-dimensional algebra.
An $A$-module $T$ is a \emph{pre-$d$-tilting module} \cite{ha}, \cite{miya} if:

\begin{enumerate}
\item $\mathrm{proj.dim}(T)\leq d$.
\item $\mathrm{Ext}^i_A(T,T)=0$ for all $0<i\leq d$.
\end{enumerate}
Then $T$ is in addition \emph{$d$-tilting} if there exists an exact sequence $$0\rightarrow A\rightarrow T_0\rightarrow T_1\rightarrow \cdots \rightarrow T_d\rightarrow 0$$ 
where $T_0,\ldots, T_d\in \mathrm{add}(T)$. 
The importance of tilting modules is highlighted by the following theorem:
\begin{theorem}[Happel]\cite{ha}
Let $A$ be a finite-dimensional algebra, $T$ a $d$-tilting $A$-module and $B:=\mathrm{End}_A(T)^\mathrm{op}$. Then the derived functor $\mathbf{R}\mathrm{Hom}_A(T,-)$ induces an equivalence of triangulated categories
$$D^b(A)\rightarrow D^b(\mathrm{End}_A(T)^\mathrm{op}).$$
\end{theorem}
The concept of a $(d+2)$-angulated category was introduced by Geiss-Keller-Oppermann in \cite{gko}. We refer there, as well as to \cite{bt}, for a definition.  
\begin{theorem}\cite[Theorem 1]{gko}\label{ko} Let $\Lambda$ be a $d$-representation-finite algebra with $d$-cluster-tilting subcategory $\mathcal{C}\subseteq \mathrm{mod}(\Lambda)$. Then there exists a $(d+2)$-angulated category $\mathcal{U}_\Lambda$ with $d$-suspension functor $\Sigma^d$ and inverse $d$-suspension functor $\Sigma^{-d}$. Any $d$-exact sequence in $\mathcal{C}$ $$0\rightarrow M_{d+1}\rightarrow \cdots \rightarrow M_1 \rightarrow M_0\rightarrow 0$$ 
induces a $(d+2)$-angle in $\mathcal{U}_\Lambda$
$$ M_{d+1}\rightarrow \cdots \rightarrow M_1 \rightarrow M_0\rightarrow \Sigma^d (M_{d+1}).$$ 
\end{theorem}
The functor $\tau_d$ acts on $\mathcal{C}$, and induces a functor $\mathbb{S}_d$ in $\mathcal{U}_\Lambda$. Oppermann-Thomas \cite[Definition 5.22]{ot} defined the \emph{$(d+2)$-angulated cluster category of $\Lambda$} to be the orbit category $$\mathcal{O}_\Lambda:=\mathcal{U}_\Lambda/\Sigma^{-d}(\mathbb{S}_d).$$

\begin{theorem}\cite[Theorem 5.2(1)]{ot}
Let $\Lambda$ be a $d$-representation-finite algebra with $d$-cluster-tilting subcategory $\mathcal{C}\subseteq\mathrm{mod}(\Lambda)$. Then $\mathcal{O}_\Lambda$ is a $(d+2)$-angulated category with $d$-suspension $[d]$. The isomorphism classes of indecomposable objects of $\mathcal{O}_\Lambda$ are in bijection with the indecomposable direct summands of $M\oplus \Lambda[d]$, where $M$ is an additive generator of $\mathcal{C}$.
\end{theorem}

The concepts of rigid and cluster-tilting objects can be extended to the language of $(d+2)$-angulated categories.
\begin{defn}\cite[Definition 5.3]{ot}
Let $\mathcal{O}_\Lambda$ be the $(d+2)$-angulated cluster category of $\Lambda$ with $d$-suspension $[d]$. An object $T\in \mathcal{O}_\Lambda$ is \emph{$d$-rigid} if $$\mathrm{Hom}_{\mathcal{O}_\Lambda}(T,T[d])=0.$$ A $d$-rigid object $T\in \mathcal{O}_\Lambda$ is a \emph{Oppermann-Thomas cluster-tilting object} if any $X\in \mathcal{O}_\Lambda$ occurs in a $(d+2)$-angle
$$X[-d]\rightarrow T_d\rightarrow\cdots\rightarrow T_1\rightarrow T_0\rightarrow X$$ with $T_i\in\mathrm{add}(T)$ for all $0\leq i\leq d$. The endomorphism algebra 
$\mathrm{End}_{\mathcal{O}_\Lambda}(T)^\mathrm{op}$ is called a \emph{Oppermann-Thomas cluster-tilted algebra}.
\end{defn}

Tilting and Oppermann-Thomas cluster-tilting objects are related as follows.
\begin{theorem}\cite[Theorem 5.5, Theorem 5.37]{ot}\label{trivext}
Let $\Lambda$ be a $d$-representation-finite algebra with canonical $d$-cluster-tilting subcategory $\mathcal{C}\subseteq\mathrm{mod}(\Lambda)$. Let $T$ be a $d$-tilting $\Lambda$-module so that $\mathrm{add}(T)\subseteq \mathcal{C}$. Then
\begin{enumerate}
\item $T$ is an Oppermann-Thomas cluster-tilting object in $\mathcal{O}_\Lambda$.
\item There is a $d$-cluster-tilting subcategory $\mathcal{D}\subseteq \mathrm{mod}(\mathrm{End}_{\mathcal{O}_\Lambda}(T)^\mathrm{op})$.
\end{enumerate}
\end{theorem}

\section{Higher gentle algebras}
An algebra $A=kQ/I$ is \emph{special biserial} if 
\begin{enumerate}
\item Each vertex of $Q$ has at most two arrows starting from it.
\item Each vertex of $Q$ has at most two arrows ending at it.
\item For each arrow $\alpha\in Q_1$, there is most one arrow $\beta$ such that $\alpha\beta\notin I$.
\item For each arrow $\gamma\in Q_1$, there is most one arrow $\beta$ such that $\beta\gamma\notin I$.
\end{enumerate}
If moreover $A$ satisfies 
\begin{enumerate}
\item The ideal $I$ is generated by paths of length at most two.
\item For each arrow $\alpha\in Q_1$, there is most one arrow $\beta$ such that $\alpha\beta\in I$.
\item For each arrow $\gamma\in Q_1$, there is most one arrow $\beta$ such that $\beta\gamma\in I$.
\end{enumerate}
then $A$ is said to be \emph{gentle}. 
In other words, on either side of each arrow in $Q$ there is at most one arrow such that the composition with this arrow is in $I$, and at most one such that the composition is not in $I$. 
 We are interested in generalising special-biserial algebras in the following fashion, the difference between this generalisation and that of special-multiserial algebras is that we allow a plentiful amount of commutativity relations, but fewer zero relations.

An algebra $A=kQ/I$ \emph{contains an $m$-cube} if there is a collection of paths between vertices $x$ and $y$ in the quiver of $A$ such that the underlying graph is an $m$-dimensional cube and any two paths defining a square face in this $m$-cube commute in $A$. For each arrow $\beta\in Q_1$, then an arrow $\alpha\in Q_1$ is a \emph{strong successor of $\beta$ in $A$} if $s(\alpha)=t(\beta)$, neither $\alpha\beta\in I$ nor are there are arrows $\alpha^\prime$ and $\beta^\prime$ and a relation $\alpha\beta-\alpha^\prime\beta^\prime\in I$. In this case $\beta$ is also a \emph{strong predecessor of $\alpha$ in $A$}.
Consider an algebra $A=kQ/I$ satisfying the following conditions, which we consider to some extent as a replacement the special-biserial conditions.
\begin{enumerate}
\item[(A1)] Each vertex of $Q$ has at most $d$ arrows starting from it. \label{fu}
\item[(A1')] Each vertex of $Q$ has at most $d$ arrows ending at it.
\item[(A2)]  For each arrow $\alpha\in Q_1$ there is at most one strong successor $\beta \in Q_1$ of $\alpha$ in $A$.
\item[(A2')]  For each arrow $\beta\in Q_1$ there is at most one strong predecessor $\alpha\in Q_1$ of $\beta$ in $A$.
\item[(A3)]  Let $\alpha\in Q_1$ be an arrow with strong successor $\beta$. Then for any $1< m<d$ and any set of $m$ arrows $\beta_i$ indexed by $1\leq i\leq m$ such that $s(\beta_i)=t(\alpha)$ and $\beta_i\alpha\notin I$ for all $1\leq i\leq m$, there is a unique $(m+1)$-cube containing $\beta$ and all $\beta_i$.\label{sb6}
\item[(A3')]  Let $\alpha\in Q_1$ be an arrow with strong predecessor $\beta$. Then for any $1< m<d$ and any set of $m$ arrows $\beta_i$ indexed by $1\leq i\leq m$ such that $s(\alpha)=t(\beta_i)$ and $\alpha\beta_i\notin I$ for all $1\leq i\leq m$, there is a unique $(m+1)$-cube containing $\beta$ and all $\beta_i$.  \label{this}
\item[(A4)] The ideal $I$ is generated by paths and commutativity relations of length two.
\end{enumerate}

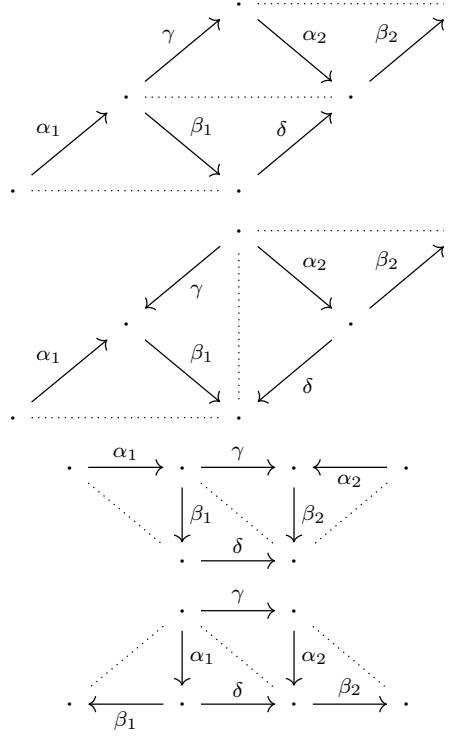
\begin{figure}
\caption{The four configurations whereby two zero relations sandwich a commutativity relation}\label{fig1}
\begin{tikzcd}
\ &\ &\cdot\arrow{dr}{\alpha_2}\arrow[-,dotted]{rr}&&\cdot\\
\ &\cdot\arrow{ur}{\gamma}\arrow{dr}{\beta_1}\arrow[-,dotted]{rr}&&\cdot\arrow{ur}{\beta_2}\\
\cdot\arrow{ur}{\alpha_1}\arrow[-,dotted]{rr}&&\cdot\arrow{ur}{\delta} 
\end{tikzcd}

\begin{tikzcd}
\ &\ &\cdot\arrow{dr}{\alpha_2}\arrow{dl}{\gamma}\arrow[-,dotted]{rr}&&\cdot\\
\ &\cdot\arrow{dr}{\beta_1}&&\cdot\arrow{ur}{\beta_2}\arrow{dl}{\delta} \\
\cdot\arrow{ur}{\alpha_1}\arrow[-,dotted]{rr}&&\cdot\arrow[-,dotted]{uu}
\end{tikzcd}

\begin{tikzcd}
\cdot\arrow{r}{\alpha_1}\arrow[-,dotted]{dr}&\cdot \arrow[-,dotted]{dr}\arrow{d}{\beta_1}\arrow{r}{\gamma}&\cdot\arrow{d}{\beta_2}&\cdot\arrow{l}{\alpha_2}\arrow[-,dotted]{dl} \\
\ &\cdot\arrow{r}{\delta}&\cdot
\end{tikzcd}

\begin{tikzcd}
\ &\cdot\arrow{d}{\alpha_1}\arrow[-,dotted]{dl}\arrow[-,dotted]{dr}\arrow{r}{\gamma}&\cdot\arrow{d}{\alpha_2}\arrow[-,dotted]{dr} \\
\cdot &\cdot\arrow{r}{\delta}\arrow{l}{\beta_1}& \cdot \arrow{r}{\beta_2}& \cdot
\end{tikzcd}
\end{figure}

We say that two zero relations $\beta_i\alpha_i\in I$, $i\in\{1,2\}$, \emph{sandwich a commutativity relation} if  there exist arrows $\gamma,\delta\in Q_1$ satisfying any of the diagrams in Figure \ref{fig1}.

\begin{defn}\label{genttle}
An algebra $A=kQ/I$ satisfying axioms (A1)-(A4) is \emph{$(d-1)$-pre-gentle} if it satisfies the following additional axioms:

\begin{enumerate}[(E1)]
\item For each arrow $\alpha\in Q_1$, there is at most one arrow $\beta$ such that $\alpha\beta\in I$.\label{E1}
\item For each arrow $\gamma\in Q_1$, there is at most one arrow $\beta$ such that $\beta\gamma \in I$.\label{E2}
\item There exists no commutativity relation that is sandwiched by zero relations.\label{E5}
\item For every idempotent $e$ of $A$, then $eAe$ satisfies axioms (E\ref{E1})-(E\ref{E5}).\label{E6}
\end{enumerate}
Finally, an algebra $B$ is \emph{$(d-1)$-gentle} if:
\begin{enumerate}
\item There is a $(d-1)$-pre-gentle algebra $A$ and an idempotent $e$ of $A$ such that $B\cong eAe$. 
\item For every idempotent $f$ of $B$, then the quiver of $fBf$ contains no $d$-cube.
\end{enumerate}
\end{defn}

A $1$-gentle algebra is simply a gentle algebra.
The axiom forbidding any sandwiching by zero relations should be thought of as a replacement for having no zero relations of length greater than two. For any gentle algebra $A$ and any idempotent $e$ of $A$, then also $eAe$ is gentle. Once we introduce commutativity relations, we need some other means by which to control the length of zero relations - this is achieved by forbidding the sandwiching of a commutativity relation by zero relations. Nevertheless, as we shall see in Example \ref{bigeg}, this is the condition that would make the most sense to relax.
 
For studying $d$-gentle algebras, it is helpful to start with the notion of a localisable object. Localisable objects were introduced in \cite{ck}: the name stems from the localising subcategories studied by Geigle and Lenzing in \cite[Section 2]{gl}. 

\begin{defn} Let $A$ be a finite-dimensional algebra. Then an object $S\in\mathrm{mod}(A)$ is a \emph{localisable} if:
\begin{itemize}
\item the module $S$ is simple,
\item  $\mathrm{proj.dim}_A(S)\leq 1$, and 
\item  $\mathrm{Ext}^1_A(S,S)=0$. 
\end{itemize}
\end{defn}

Every localisable object $S$ can be expressed as $S\cong A/\langle f\rangle$ for some idempotent $f$ in $A$, since $\mathrm{Ext}^1_A(S,S)=0$. This was generalised in the following sense in \cite{mc2}:

\begin{defn}\label{fabdefa}
	Let $A$ be a finite-dimensional algebra and $e, f$ idempotents of $A$. Then $f$ is an \emph{fabric idempotent of $A$ with respect to $e$} (or simply $f$ is a \emph{fabric idempotent}) if:
	\begin{itemize}
		\item the idempotent $f$ satisfies $\mathrm{proj.dim}_A(A/\langle f \rangle) \leq 1$.
		\item For every projective $A/\langle f \rangle$-module $P$, the module $\tau_A(P)$ is injective as an \\$A/\langle e \rangle$-module.
\item For every injective $A/\langle e \rangle$-module $I$, the module $\tau_A^{-1}(I)$ is projective as an \\$A/\langle f \rangle$-module.			\end{itemize}
\end{defn}

\begin{theorem}\cite[Theorem 2.1]{chen4}  \cite[Corollary 3.3]{chen3}\cite[Theorem 5.2]{pss}\label{chensing}
Let $A$ be a finite-dimensional algebra and $f$ an idempotent of $A$. Then there is an equivalence $$D_{\mathrm{sg}}(A)\cong D_{\mathrm{sg}}(fAf)$$ 
if and only if the algebra $A$ satisfies $\mathrm{proj.dim}_{fAf}(fA)<\infty$ and $\mathrm{proj.dim}_A(M)<\infty$ for all modules $M\in \mathrm{mod}(A/\langle f \rangle)$. 
\end{theorem}
Fabric idempotents are useful because of the following corollary.
\begin{cor}\cite[Corollary 3.7]{mc2}\label{fabthm}
Let $A$ be a finite-dimensional algebra with fabric idempotent $f$. If in addition $\mathrm{gl.dim}(A/\langle f\rangle)<\infty$, then there is an equivalence $$D_{\mathrm{sg}}(A)\cong D_{\mathrm{sg}}(fAf).$$  
\end{cor}
By design, we obtain the following result.
\begin{theorem}\label{elso}
Let $A$ be a $d$-gentle algebra. Then there exists an idempotent $f$ such that $f$ is a product of fabric idempotents, $fAf$ is a gentle algebra and there is an equivalence of categories $$D_{\mathrm{sg}}(A)\cong D_{\mathrm{sg}}(fAf).$$
\end{theorem}
\begin{proof}
We first prove the result for $d$-pre-gentle algebras. The proof is similar to the case for higher Nakayama algebras \cite[Theorem 4.4]{mc2}.
Suppose there are four vertices $a,b,c,d\in Q_0$ and non-zero paths $w_1,w_2,w_3,w_4$ in $Q$ as follows:
	$$
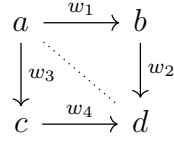
\begin{tikzcd}
a\arrow{r}{w_1}\arrow[dotted, dash]{dr}\arrow{d}{w_3}&b\arrow{d}{w_2}\\
c \arrow{r}{w_4}&d
\end{tikzcd}$$
Suppose that there does not exist a surjective morphism $I_b\twoheadrightarrow I_a$. Then there is a vertex $x$ and a path $w_x:x\rightarrow a$ such that $w_1w_x\in I$ as in Figure \ref{picture}.

\begin{figure}[h]\caption{The case of a commutative square with no surjection $I_b\twoheadrightarrow I_a$.}\label{picture}
\begin{tikzcd}
x\arrow{d}{w_x} \arrow[dotted, dash]{dr}\\
a\arrow{r}{w_1}\arrow{d}{w_3}\arrow[dotted, dash]{dr}&b\arrow{d}{w_2} \\
c \arrow{r}{w_4}&d
\end{tikzcd}
\end{figure}
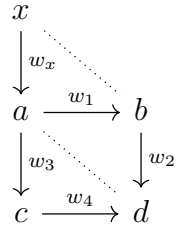
There must be an inclusion $P_c\hookrightarrow P_a$, otherwise suppose there is a vertex $g$ and an path $w_g:c\rightarrow g$ such that $w_gw_3\in I$:
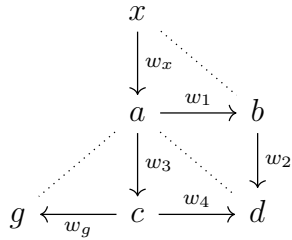
$$\begin{tikzcd}
\ &x\arrow{d}{w_x} \arrow[dotted, dash]{dr}\\
\ &a\arrow{r}{w_1}\arrow[dotted,-]{dl}\arrow{d}{w_3}\arrow[dotted, dash]{dr}&b\arrow{d}{w_2}\\
g&c\arrow{l}{w_g} \arrow{r}{w_4}&d 
\end{tikzcd}$$
Then we obtain the following forbidden diagram:
$$
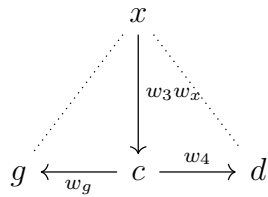
\begin{tikzcd}
\ &x\arrow{dd}{w_3w_x} \arrow[dotted, dash]{ddr}\arrow[dotted, dash]{ddl}\\
\\
g&c\arrow{l}{w_g} \arrow{r}{w_4}&d 
\end{tikzcd}$$
Secondly, there must be an inclusion $P_d\hookrightarrow P_c$, otherwise there is a vertex $h$ and a path $w_h:d\rightarrow h$ such that $w_hw_4\in I$. This implies a sandwiching by zero relations:

$$
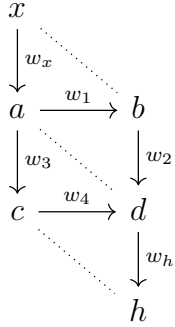
\begin{tikzcd}
x\arrow{d}{w_x} \arrow[dotted, dash]{dr}\\
a\arrow{r}{w_1}\arrow{d}{w_3}\arrow[dotted, dash]{dr}&b\arrow{d}{w_2} \\
c \arrow{r}{w_4}\arrow[dotted, dash]{dr}&d\arrow{d}{w_h}\\
\ &h
\end{tikzcd}$$

So there must be injective morphisms $P_d\hookrightarrow P_c$ and $P_c\hookrightarrow P_a$. Define the module $M:=\mathrm{coker}(P_d\rightarrow P_c)$, and choose the smallest idempotent $f$ of $A$ such that $M\in\mathrm{proj}(A/\langle f\rangle)$. Then $f$ is a fabric idempotent by \cite[Proposition 3.2]{mc2}. Since there is an inclusion $P_c\hookrightarrow P_a$, there cannot exist a cycle containing $c$, and similarly nor any other vertex supported by $M$. This implies that $\mathrm{gl.dim}(A/\langle f\rangle)\leq \infty$.  Hence Corollary \ref{fabthm} implies that $D_{\mathrm{sg}}(A)\cong D_\mathrm{sg}(fAf)$.

A similar argument can be made if there does not exist a surjective morphism $I_b\twoheadrightarrow I_a$, and a dual argument may be made if there exists neither an injective morphism $P_d\hookrightarrow P_c$ nor an injective morphism $P_c\hookrightarrow P_a$.

Since zero relations are of length two, we may remove all commutativity relations for $A$ as above, and this will not result in any zero relations of length greater than three (compare with Figure \ref{picture} without vertex $c$). Hence we are reducing to a $d$-gentle algebra whose singularity category is equivalent to that of $A$, and that has no commutativity relations and no zero relations of length greater than two. In other words a gentle algebra.

Now suppose we have an arbitrary $d$-gentle algebra $eAe$ (where $A$ is a $d$-pre-gentle algebra), and a zero relation of length greater than two. It should be easily seen that any such relation may only occur if there were a series of commutativity relations in $A$ neighbouring a relation of length two, one case of which is presented in Figure \ref{picture}. Consider the case of Figure \ref{picture}, and assume $e=e_x+e_a+e_c+e_d$. So we see that our construction is precisely compatible for $d$-gentle algebras - we obtain fabric idempotent $f=e_x+e_a+e_d$ of $eAe$, which is now a relation of length two. So by applying this reduction technique as above then we again arrive at a gentle algebra.
\end{proof}

Singularity categories for gentle algebras were described in \cite{ka}, hence we may also describe the singularity category of any $d$-gentle algebra. Singularity categories are especially interesting for Iwanaga-Gorenstein algebras, because of Theorem \ref{buch}.  The following result generalises the main theorem of \cite{gr}.

\begin{cor}
Any $d$-gentle algebra is Iwanaga-Gorenstein.
\end{cor}

\begin{proof}
Given a $d$-gentle algebra $A$, consider an infinite projective resolution (as a complex) of some injective $A$-module $I$ $$\cdots\rightarrow P_n\rightarrow \cdots \rightarrow P_1\rightarrow P_0\rightarrow 0.$$
By Theorem \ref{elso}, there is a product of fabric idempotents $f$, such that $fAf$ is gentle. Since $f$ is a product of fabric idempotents, neither $I$ nor $\Omega^i(I)$ are in $\mathrm{mod}(A/\langle f\rangle)$ for any $i\geq 1$. 
Applying the functor $\mathrm{Hom}_A(Af,-)$ induces the exact sequence of projective $fAf$-modules $$\cdots\rightarrow fP_n\rightarrow \cdots \rightarrow fP_1\rightarrow fP_0\rightarrow 0.$$ This is an infinite projective resolution of the injective $fAf$-module $fI$, and so $fI$ has infinite projective dimension. Since $fAf$ is gentle and hence Iwanaga-Gorenstein, this is a contradiction.
\end{proof}

\section{Tilted algebras of higher Auslander algebras of type $\vec{A}$}
Following the notation of Oppermann-Thomas \cite[Definition 2.2]{ot}, define the sets $$\mathbf{I}_{m}^d:=\{(i_0,\ldots,i_d)\in \{1,\ldots,m\}^{d+1}|\forall x\in \{0,1,\ldots,d-1\}:i_x+2\leq i_{x+1}\},$$
$$^\circlearrowleft\mathbf{I}_{m}^d:=\{(i_0,\ldots,i_d)\in \mathbf{I}_{m}^d|i_d+2\leq i_0+m\}.$$

Given two increasing $(d+1)$-tuples of real numbers $X=\{x_0,x_1,\ldots,x_d\}$ and $Y=\{y_0,y_1,\ldots,y_d\}$ , then \emph{$X$ intertwines $Y$} if $x_0<y_0<x_1<y_1<\cdots <x_d<y_d$. Denote by $X\wr Y$ if $X$ intertwines $Y$. A collection of increasing $(d+1)$-tuples of real numbers is \emph{non-intertwining} if no pair of elements intertwine (in either order). 

In \cite{iy2}, Iyama describes an inductive construction of an $(i+1)$-Auslander algebra from an $i$-Auslander algebra. In particular, for the path algebra of type $\vec{A}_n$, a $d$-Auslander algebra may be constructed, we call this algebra $A^d_n$. As part of this construction, the category $\mathrm{mod}(A^{d}_n)$ has a canonical $d$-cluster-tilting subcategory, and it is unique by Theorem \ref{artheory}. By \cite[Theorem 3.4]{ot}, the indecomposable modules in this $d$-cluster-tilting subcategory, as well as the vertices of the quiver of $A^{d+1}_n$ may be labelled by $\mathbf{I}_{n+2d}^{d}$.
The quiver of $A^2_4$ is as follows:

$$
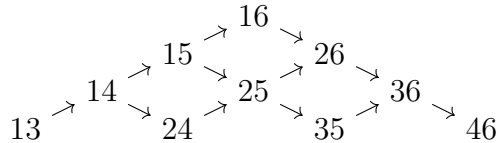
\begin{tikzpicture}[xscale=5,yscale=2.5]
\node(xa) at (-2,1){$13$};
\node(xb) at (-1.6,1){$24$};
\node(xc) at (-1.2,1){$35$};
\node(xd) at (-0.8,1){$46$};

\node(xab) at (-1.8,1.2){$14$};
\node(xac) at (-1.4,1.2){$25$};
\node(xad) at (-1,1.2){$36$};

\node(xbb) at (-1.6,1.4){$15$};
\node(xbc) at (-1.2,1.4){$26$};

\node(xcc) at (-1.4,1.6){$16$};

\draw[->](xa) edge(xab);
\draw[->](xb) edge(xac);
\draw[->](xc) edge(xad);

\draw[->](xbb) edge(xac);
\draw[->](xbc) edge(xad);

\draw[->](xbb) edge(xcc);

\draw[->](xab) edge(xb);
\draw[->](xac) edge(xc);
\draw[->](xad) edge(xd);

\draw[->](xab) edge(xbb);
\draw[->](xac) edge(xbc);

\draw[->](xcc) edge(xbc);
\end{tikzpicture}$$

The quiver of $A^3_4$ is as follows:

$$
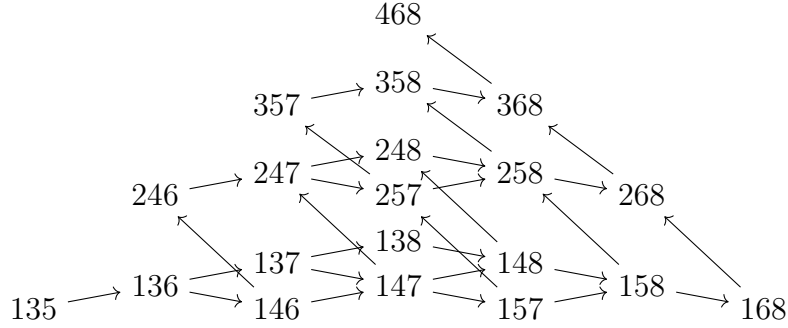
\begin{tikzpicture}[xscale=8,yscale=1.5]

\node(a) at (-2,0){$135$};
\node(b) at (-1.6,0){$146$};
\node(c) at (-1.2,0){$157$};
\node(d) at (-0.8,0){$168$};
\node(ab) at (-1.8,0.2){$136$};
\node(ac) at (-1.4,0.2){$147$};
\node(ad) at (-1,0.2){$158$};
\node(bb) at (-1.6,0.4){$137$};
\node(bc) at (-1.2,0.4){$148$};
\node(cc) at (-1.4,0.6){$138$};
\draw[->](a) edge (ab);
\draw[->](b) edge (ac);
\draw[->](c) edge (ad);
\draw[->](bb) edge (ac);
\draw[->](bc) edge (ad);
\draw[->](bb) edge (cc);
\draw[->](ab) edge (b);
\draw[->](ac) edge (c);
\draw[->](ad) edge (d);
\draw[->](ab) edge (bb);
\draw[->](ac) edge (bc);
\draw[->](cc) edge (bc);
\node(xa) at (-1.8,1){$246$};
\node(xb) at (-1.4,1){$257$};
\node(xc) at (-1,1){$268$};
\node(xab) at (-1.6,1.2){$247$};
\node(xac) at (-1.2,1.2){$258$};
\node(xbb) at (-1.4,1.4){$248$};
\draw[->](xa) edge(xab);
\draw[->](xb) edge(xac);
\draw[->](xbb) edge(xac);
\draw[->](xab) edge(xb);
\draw[->](xac) edge(xc);
\draw[->](xab) edge(xbb);
\node(xxa) at (-1.6,1.8){$357$};
\node(xxb) at (-1.2,1.8){$368$};
\node(xxab) at (-1.4,2){$358$};
\draw[->](xxa) edge(xxab);
\draw[->](xxab) edge(xxb);
\node(xxxa) at (-1.4,2.6){$468$};
\draw[->](b) edge (xa);
\draw[->](c) edge (xb);
\draw[->](d) edge (xc);
\draw[->](ac) edge (xab);
\draw[->](ad) edge (xac);
\draw[->](bc) edge (xbb);

\draw[->](xb) edge (xxa);
\draw[->](xc) edge (xxb);
\draw[->](xac) edge (xxab);

\draw[->](xxb) edge (xxxa);

\end{tikzpicture}$$
 For each $I\in\mathbf{I}_{n+2d}^{d}$, denote by $M_I$ the object of the aforementioned $d$-cluster-tilting subcategory, and let $M$ be an additive generator of the subcategory. Then there is a combinatorial description of tilting $A^d_n$-modules.

\begin{theorem}\cite[ Theorem 3.6(4), Theorem 4.4]{ot}\label{to}
Let $I,J\in\mathbf{I}_{n+2d}^{d}$. Then $\mathrm{Ext}^d_{A^d_n}(M_I,M_J)\ne 0\ \iff J\wr I$. Moreover, there are bijections between 
\begin{itemize}
\item triangulations of the cyclic polytope $\mathrm{C}(n+2d,2d)$.
\item non-intertwining collections of ${n+d-1}\choose {d}$ $(d+1)$-tuples in $\mathbf{I}_{n+2d}^d$.
\item isomorphism classes of summands of $_{A^d_n}M$ which are tilting modules. 
\end{itemize}
\end{theorem}

We refer to \cite{ot} for a definition of cyclic polytopes, as this is beyond the scope of this article. 
\begin{theorem}\cite[Lemma 6.6, Proposition 6.1, Theorem 6.4]{ot}\label{too}
Consider the $d$-representation-finite algebra $A_n^d$, and $\mathcal{O}_{A_n^d}$, the $(d+2)$-angulated cluster category of $A^d_n$. Then:
\begin{enumerate}
\item the indecomposable objects of the $(d+2)$-angulated category $\mathcal{O}_{A_n^d}$ are indexed by $^\circlearrowleft\mathbf{I}_{n+2d+1}^d$.
\item
$\mathrm{Hom}_{\mathcal{O}_{A_n^d}}(M_I,M_J[d])\ne 0\iff I\wr J \text{ or } J\wr I$.
\item Triangulations of the cyclic polytope $C(n+2d+1,2d)$ correspond bijectively to basic Oppermann-Thomas cluster-tilting objects in $\mathcal{O}_{A_n^d}$. 
\end{enumerate}
\end{theorem}

The following result may now be obtained.
\begin{cor}\label{boo}
Let $T$ be any $d$-rigid $A^d_n$-module in $\mathcal{C}$, where $\mathcal{C}\subseteq \mathrm{mod}(A^d_n)$ is the canonical $d$-cluster-tilting subcategory. Let $B=\mathrm{End}_{A^d_n}(T)^\mathrm{op}$. Then $B$ is a $d$-gentle algebra.
\end{cor}
\begin{proof}
Let $T$ be a $d$-rigid $A^d_n$-module in $\mathcal{C}$. By Theorem \ref{too}, this corresponds to a set $\mathcal{I}$ of non-intertwining subsets of $\mathbf{I}_{n+2d}^{d}$. 
It is clear that $A^d_n$ is $d$-pre-gentle. Let $e$ be an idempotent of $A^d_n$ corresponding to $\mathcal{I}$, then $eA^d_ne\cong B$. Finally, $eA^d_ne$ contains no $d$-cube, owing precisely to the $d$-rigid condition.
\end{proof}
We show that in some cases Oppermann-Thomas cluster-tilted algebras of type $A^d_n$ are $d$-gentle. Let $1\leq i< n$, and let the subset given by $\{i,i+2,\ldots, i+2(d-1)\}$ (modulo $n$) be denoted by $I_i$. 

\begin{prop}\label{ctgent}
Let $\mathcal{I}\subset \{1,2,\ldots, n\}$ be a subset such that $i,j\in\mathcal{I}$ implies $i\ne j+1 (\mathrm{mod}\ n)$. Let $$T=\bigoplus_{\substack{x\in ^\circlearrowleft\mathbf{I}_{n+2d-1}^{d-1}\\ x\ne I_i\ \text{for any}\ i\in\mathcal{I}}}P_x\oplus \bigoplus_{i\in I} \tau_d^{-1} S_{I_i}.$$ Then $T$ is a $d$-tilting $A^d_n$-module, and $B:=\mathrm{End}_{\mathcal{O}_{A^d_n}}(T)^\mathrm{op}$ is a $d$-gentle algebra. 
\end{prop}

\begin{proof}
By \cite[Theorem 2.3.1]{iy1}, for any algebra $\Lambda$ with global dimension $d$ and any two $\Lambda$-modules $M$ and $N$, there is an isomorphism $$\mathrm{Hom}_\Lambda(M,\tau_d(N))\cong \mathrm{Ext}^d_\Lambda(N,M).$$ So $\mathrm{Ext}_{A^d_n}^d(T,T)\cong \mathrm{Hom}_{A^d_n}(T,\oplus_{i\in\mathcal{I}}S_{I_i})=0$. Observe that $\tau_d^{-1}S_{I_i}\cong S_{I_{i-1}}$. It is straightforward to see for any $i\in\mathcal{I}$ that there is an exact sequence:
$$0\rightarrow P_{I_i}\rightarrow P_{i-1,i+2,i+4,\ldots, i+2d-2}\rightarrow P_{i-1, i+1,i+4,\ldots,i+2d-2}\rightarrow \cdots\rightarrow P_{I_{i-1}}\rightarrow  S_{I_{i-1}} \rightarrow 0$$and hence $T$ is a $d$-tilting module.

Let $\mathcal{I}$ be as above, and let $A$ be the algebra $\mathrm{End}_{\mathcal{O}_{A^d_n}}(A^d_n\bigoplus_{i\in I} \tau_d^{-1} S_{I_i})^\mathrm{op}$. By construction, the zero relations are determined by the $I_i$ for $1\leq i\leq n$, and it can be easily seen that there can be no sandwiching by zero relations. Since every relation in $A$ is of length two, $A$ is a $d$-pre-gentle algebra. By definition, there is an idempotent $e$ such that $B=eAe$, and since $T$ is $d$-rigid there can be no $d$-cube in $B$. 
\end{proof}
More generally, we may state the following.
\begin{cor}\label{simplytoogood}
Let $\mathcal{C}\subseteq\mathrm{mod}(A^d_n)$ be the canonical $d$-cluster-tilting subcategory and let $S$ be a semisimple $A^d_n$-module in $\mathcal{C}$. Suppose that $\mathrm{Ext}^d_{A^d_n}(S,S)=0$ and let $P$ be a basic projective $A^d_n$-module such that $\mathrm{Ext}^d_{A^d_n}(S,P)=0$. For $T^\prime:=P\oplus \tau^{-1}_d(S)$, the algebra $\mathrm{End}_{\mathcal{O}_{A^d_n}}(T^\prime)^\mathrm{op}$ is $d$-gentle.
\end{cor}

\begin{proof}
Let $\mathcal{I}\subset \{1,2,\ldots, n-1\}$ be a subset such that $i,j\in\mathcal{I}$ implies $i\ne j+1 (\mathrm{mod}\ n)$. There is a bijection between such subsets and semisimple non-projective $A^d_n$-modules in $\mathcal{C}$ such that $\mathrm{Ext}^d_{A^d_n}(S,S)=0$.
Let $T$ be defined as in the statement of Proposition \ref{ctgent}, then $\mathrm{End}_{\mathcal{O}_{A^d_n}}(T)^\mathrm{op}$ is a $d$-gentle algebra. So choose any projective module $P$ such that $\mathrm{Ext}^d_{A^d_n}(S,P)=0$. Then there must be an idempotent $e$ such that  $$\mathrm{End}_{\mathcal{O}_{A^d_n}}(T^\prime)^\mathrm{op}=e \mathrm{End}_{\mathcal{O}_{A^d_n}}(T)^\mathrm{op}e $$ and this must determine a $d$-gentle algebra. 
\end{proof}

\section{Examples}\label{sec3}

Consider the non-intertwining collection of $3$-tuples in $^\circlearrowleft\mathbf{I}^2_{10}$ $$\mathcal{I}:=\{135,136,137,138,139,147,148,149,157,158,159,169,179,357,579\}.$$ Then $\mathcal{I}$ corresponds to an Oppermann-Thomas cluser-tilting $A^2_5$-module $T$ by Theorem \ref{to} and Theorem \ref{too}. The algebra of $B:=\mathrm{End}_{\mathcal{O}_{A^2_5}}(T)^\mathrm{op}$ is as follows. 

$$
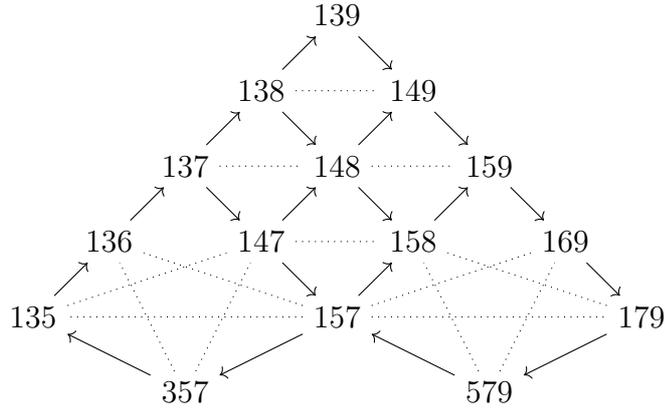
\begin{tikzpicture}[xscale=5,yscale=5]
\node(a) at (-2,0){$135$};
\node(b) at (-1.6,-0.2){$357$};
\node(c) at (-1.2,0){$157$};
\node(d) at (-0.8,-0.2){$579$};
\node(e) at (-0.4,0){$179$};

\node(ab) at (-1.8,0.2){$136$};
\node(ac) at (-1.4,0.2){$147$};
\node(ad) at (-1,0.2){$158$};
\node(ae) at (-0.6,0.2){$169$};

\node(bb) at (-1.6,0.4){$137$};
\node(bc) at (-1.2,0.4){$148$};
\node(bd) at (-0.8,0.4){$159$};

\node(cc) at (-1.4,0.6){$138$};
\node(cd) at (-1,0.6){$149$};

\node(dd) at (-1.2,0.8){$139$};

\draw[->](a) edge (ab);
\draw[->](c) edge (ad);

\draw[->](bb) edge (ac);
\draw[->](bc) edge (ad);
\draw[->](bd) edge (ae);

\draw[->](bb) edge (cc);
\draw[->](bc) edge (cd);

\draw[->](cc) edge (dd);

\draw[->](ac) edge (c);
\draw[->](ae) edge (e);

\draw[->](ab) edge (bb);
\draw[->](ac) edge (bc);
\draw[->](ad) edge (bd);

\draw[->](cc) edge (bc);
\draw[->](cd) edge (bd);

\draw[->](dd) edge (cd);

\draw[->](e) edge (d);
\draw[->](d) edge (c);
\draw[->](c) edge (b);
\draw[->](b) edge (a);

\draw[-,dotted](bb) edge (bc);
\draw[-,dotted](ac) edge (ad);
\draw[-,dotted](bc) edge (bd);
\draw[-,dotted](cc) edge (cd);

\draw[-,dotted](b) edge (ab);
\draw[-,dotted](b) edge (ac);
\draw[-,dotted](d) edge (ad);
\draw[-,dotted](d) edge (ae);

\draw[-,dotted](a) edge (ac);
\draw[-,dotted](c) edge (ab);
\draw[-,dotted](e) edge (ad);
\draw[-,dotted](c) edge (ae);

\draw[-,dotted](c) edge (e);

\draw[-,dotted](c) edge (a);

\end{tikzpicture}$$

This is a $2$-gentle algebra in the setting of Proposition \ref{ctgent}. The proof of Theorem \ref{elso} shows that the singularity category of $B$ is equivalent to that of the following gentle algebra:

$$
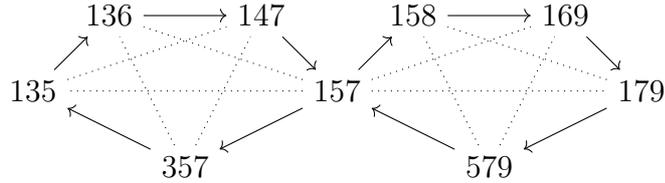
\begin{tikzpicture}[xscale=5,yscale=5]

\node(a) at (-2,0){$135$};
\node(b) at (-1.6,-0.2){$357$};
\node(c) at (-1.2,0){$157$};
\node(d) at (-0.8,-0.2){$579$};
\node(e) at (-0.4,0){$179$};

\node(ab) at (-1.8,0.2){$136$};
\node(ac) at (-1.4,0.2){$147$};
\node(ad) at (-1,0.2){$158$};
\node(ae) at (-0.6,0.2){$169$};

\draw[->](a) edge (ab);
\draw[->](c) edge (ad);

\draw[->](ac) edge (c);
\draw[->](ae) edge (e);

\draw[->](e) edge (d);
\draw[->](d) edge (c);
\draw[->](c) edge (b);
\draw[->](b) edge (a);

\draw[->](ab) edge (ac);
\draw[->](ad) edge (ae);

\draw[-,dotted](b) edge (ab);
\draw[-,dotted](b) edge (ac);
\draw[-,dotted](d) edge (ad);
\draw[-,dotted](d) edge (ae);

\draw[-,dotted](a) edge (ac);
\draw[-,dotted](c) edge (ab);
\draw[-,dotted](e) edge (ad);
\draw[-,dotted](c) edge (ae);
\draw[-,dotted](c) edge (e);

\draw[-,dotted](c) edge (a);

\end{tikzpicture}$$

It is unfortunately not true that for every tilting $A_n^d$-module $T$, the algebra $\mathrm{End}_{\mathcal{O}_{A^d_n}}(T)^\mathrm{op}$ is $d$-gentle. Nevertheless, this does not mean that singularity categories for such algebras are difficult to calculate. 
\begin{eg}\label{bigeg}
An example is the following algebra $A$, which corresponds to the maximal non-intertwining collection $$\mathcal{I}:=\{135,136,137,138,148,158,168, 357,358,368\}.
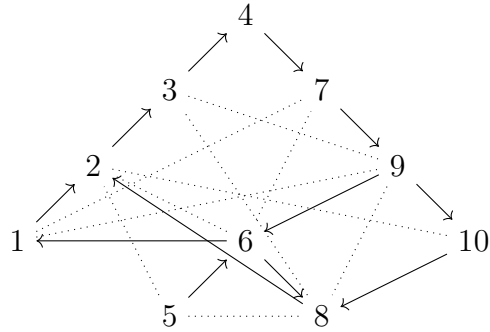
$$
$$
\begin{tikzpicture}[xscale=5,yscale=5]

\node(a) at (-2,0){$1$};
\node(b) at (-1.6,-0.2){$5$};
\node(c) at (-1.2,-0.2){$8$};
\node(d) at (-0.8,-0){$10$};

\node(ab) at (-1.8,0.2){$2$};
\node(ac) at (-1.4,0){$6$};
\node(ad) at (-1,0.2){$9$};

\node(bb) at (-1.6,0.4){$3$};
\node(bc) at (-1.2,0.4){$7$};

\node(cc) at (-1.4,0.6){$4$};

\draw[->](a) edge (ab);
\draw[->](b) edge (ac);

\draw[->](bc) edge (ad);

\draw[->](bb) edge (cc);

\draw[->](ac) edge (c);
\draw[->](ad) edge (d);

\draw[->](ab) edge (bb);

\draw[->](cc) edge (bc);

\draw[->](ad) edge (ac);
\draw[->](d) edge (c);

\draw[->](ac) edge (a);
\draw[->](c) edge (ab);

\draw[-,dotted](a) edge (bc);
\draw[-,dotted](ac) edge (bc);
\draw[-,dotted](bb) edge (ad);
\draw[-,dotted](a) edge (ad);
\draw[-,dotted](ab) edge (d);
\draw[-,dotted](b) edge (c);
\draw[-,dotted](ad) edge (c);
\draw[-,dotted](ac) edge (ab);
\draw[-,dotted](c) edge (bb);
\draw[-,dotted](b) edge (ab);
\end{tikzpicture}$$
Then it may be calculated that $A$ is $2$-Iwanaga-Gorenstein, but not 2-gentle; there is a sandwiching by zero relations:

$$
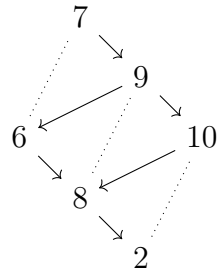
\begin{tikzpicture}[xscale=4,yscale=4]
\node(c) at (-1.2,-0.2){$8$};
\node(d) at (-0.8,-0){$10$};

\node(ab) at (-1,-0.4){$2$};
\node(ac) at (-1.4,0){$6$};
\node(ad) at (-1,0.2){$9$};

\node(bc) at (-1.2,0.4){$7$};

\draw[->](bc) edge (ad);

\draw[->](ac) edge (c);
\draw[->](ad) edge (d);

\draw[->](ad) edge (ac);
\draw[->](d) edge (c);

\draw[->](c) edge (ab);

\draw[-,dotted](ac) edge (bc);
\draw[-,dotted](ab) edge (d);
\draw[-,dotted](ad) edge (c);
\end{tikzpicture}$$
The algebra $A$ has singularity category: 

$$
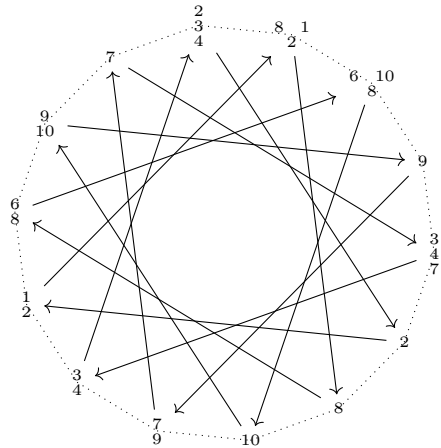
\begin{tikzpicture}[x=1.0cm,y=1.0cm]
\draw [dotted] (1.42,4.76)-- (0.26,4.34);
\draw [dotted] (0.26,4.34)-- (-0.6028927163374309,3.4582879381046174);
\draw [dotted] (-0.6028927163374309,3.4582879381046174)-- (-0.9977716638569896,2.289497701661035);
\draw [dotted] (-0.9977716638569896,2.289497701661035)-- (-0.8464262241299485,1.065122531268468);
\draw [dotted] (-0.8464262241299485,1.065122531268468)-- (-0.1788322176242585,0.027664945399204832);
\draw [dotted] (-0.1788322176242585,0.027664945399204832)-- (0.8727851751684443,-0.6173938574830413);
\draw [dotted] (0.8727851751684443,-0.6173938574830413)-- (2.1001402323646245,-0.742292070340393);
\draw [dotted] (2.1001402323646245,-0.742292070340393)-- (3.2601402323646242,-0.3222920703403931);
\draw [dotted] (3.2601402323646242,-0.3222920703403931)-- (4.123032948702056,0.5594199915549893);
\draw [dotted] (4.123032948702056,0.5594199915549893)-- (4.517911896221614,1.7282102279985714);
\draw [dotted] (4.517911896221614,1.7282102279985714)-- (4.366566456494573,2.9525853983911388);
\draw [dotted] (4.366566456494573,2.9525853983911388)-- (3.6989724499888834,3.9900429842604015);
\draw [dotted] (3.6989724499888834,3.9900429842604015)-- (2.647355057196181,4.6351017871426485);
\draw [dotted] (2.647355057196181,4.6351017871426485)-- (1.42,4.76);
\begin{scriptsize}
\node(a) at (1.42,4.76) {$\begin{smallmatrix} 2\\3\\4\end{smallmatrix}$};;
\node(b) at  (0.26,4.34) {$\begin{smallmatrix} 7\end{smallmatrix}$};
\node(c) at  (-0.6028927163374309,3.4582879381046174) {$\begin{smallmatrix} 9\\10\end{smallmatrix}$};
\node(d) at  (-0.9977716638569896,2.289497701661035) {$\begin{smallmatrix} 6\\8\end{smallmatrix}$};
\node(e) at  (-0.8464262241299485,1.065122531268468) {$\begin{smallmatrix} 1\\2\end{smallmatrix}$};
\node(f) at  (-0.1788322176242585,0.027664945399204832) {$\begin{smallmatrix} 3\\4 \end{smallmatrix}$};
\node(g) at  (0.8727851751684443,-0.6173938574830413) {$\begin{smallmatrix} 7\\9\end{smallmatrix}$};
\node(h) at  (2.1001402323646245,-0.742292070340393) {$\begin{smallmatrix} 10\end{smallmatrix}$};
\node(i) at  (3.2601402323646242,-0.3222920703403931) {$\begin{smallmatrix} 8\end{smallmatrix}$};
\node(j) at  (4.123032948702056,0.5594199915549893) {$\begin{smallmatrix} 2\end{smallmatrix}$};
\node(k) at  (4.517911896221614,1.7282102279985714) {$\begin{smallmatrix} 3\\4\\7\end{smallmatrix}$};
\node(l) at  (4.366566456494573,2.9525853983911388) {$\begin{smallmatrix} 9\end{smallmatrix}$};
\node(m) at  (3.6989724499888834,3.9900429842604015) {$\begin{smallmatrix} 6\ \ 10\\8\end{smallmatrix}$};
\node(n) at  (2.647355057196181,4.6351017871426485) {$\begin{smallmatrix} 8\ \ 1\\2\end{smallmatrix}$};

\draw[->](a) edge (j);
\draw[->](b) edge (k);
\draw[->](c) edge (l);
\draw[->](d) edge (m);
\draw[->](e) edge (n);
\draw[->](f) edge (a);
\draw[->](g) edge (b);
\draw[->](h) edge (c);
\draw[->](i) edge (d);
\draw[->](j) edge (e);
\draw[->](k) edge (f);
\draw[->](l) edge (g);
\draw[->](m) edge (h);
\draw[->](n) edge (i);
\end{scriptsize}
\end{tikzpicture}$$

where the dotted lines denote the $\Omega$ orbit, and the composition of any two arrows is zero.
\end{eg}

\section{Acknowledgements}
This paper was completed as part of my PhD studies, with the support of the Austrian Science Fund (FWF): W1230. I would like to thank my supervisor, Karin Baur, for her continued help and support during my studies, as well as Ana Garcia Elsener for useful discussions.

\bibliographystyle{amsplain}
\bibliography{genteel}

\end{document}